\documentclass[a4paper,12pt]{article}

\usepackage{setspace}
\usepackage[english]{babel}     
\usepackage[latin1]{inputenc}   
\usepackage[T1]{fontenc}        

\usepackage{vmargin}
\usepackage{color}
\usepackage{epsfig}

\usepackage{amsmath}            
\usepackage{amsthm}             
\usepackage{amsfonts}           
\usepackage{float}              
 
\usepackage{hyperref}

\setmarginsrb{3cm}{3cm}{3cm}{3cm}{0.5cm}{1cm}{0.5cm}{0.8cm}

\newtheoremstyle{th}
{4pt}{5pt}      
{\it}           
{}              
{\bf }          
{:}             
{.5em}          
{}              

\theoremstyle{th} 
\newtheorem{thm}{Theorem}[section]

\newtheorem{defi}[thm]{Definition}

\newtheorem{rmk}[thm]{Remark}

\newtheorem{lemma}[thm]{Lemma}

\newtheorem{prop}[thm]{Proposition}

\newtheorem{fact}[thm]{Fact}

\newtheoremstyle{ex}
{4pt}{5pt}      
{}              
{}              
{\bf }          
{:}             
{.5em}          
{}              

\theoremstyle{ex}

\newcommand{\Z}{\mathbb{Z}}

\newcommand{\R}{\mathbb{R}}

\newcommand{\F}{\mathbb{F}}

\newcommand{\abs}[1]{\ensuremath{\left|#1\right|}}

\date{}

\begin{document}
\author{Rizos Sklinos}
\title{On ampleness and pseudo-Anosov homeomorphisms in the free group}

\maketitle

\begin{abstract}
We use pseudo-Anosov homeomorphisms of surfaces in order to prove that the first order theory of non abelian free groups, 
$T_{fg}$, is $n$-ample for any $n\in\omega$. This result adds to the work of Pillay, 
that proved that $T_{fg}$ is non $CM$-trivial. 
The sequence witnessing ampleness is a sequence of primitive elements in $\F_{\omega}$.

Our result provides an alternative proof to the main result of a preprint by Ould Houcine-Tent \cite{OuldTentAmple}.

We also add an appendix in which we make a few remarks on Sela's paper on imaginaries in torsion 
free hyperbolic groups \cite{SelIm}. In particular we give alternative transparent proofs concerning the non-elimination of certain imaginaries. 
\end{abstract}

\section{Introduction}

The notion of $n$-ampleness, for some natural number $n$, fits in the general context of geometric stability theory. As the definition 
may look artificial or technical, we first give the historical background of its 
development. We start by working in a vector space $V$ and we consider two finite dimensional subspaces $V_1,V_2\leq V$. Then one can see 
that $dim(V_1 + V_2) = dim(V_1) + dim(V_2) - dim(V_1\cap V_2)$, and the point really is that $V_1$ is linearly independent from 
$V_2$ over $V_1\cap V_2$. In an abstract stable theory the notion of linear independence is replaced by forking independence (see Section \ref{2}) and the 
above property gives rise to the notion of $1$-basedness. A stable theory $T$ is {\em $1$-based} if there are no $a$, 
$b$ such that $acl^{eq}(a)\cap acl^{eq}(b)=acl^{eq}(\emptyset)$ and $a$ forks with $b$ over $\emptyset$. 
The notion of $1$-basedness turned out to be very fruitful in model theory and one of the major 
results concerning this notion was the following theorem by Hrushovski-Pillay \cite{WeakNormal}.

\begin{thm}\label{HruPill}
Let $\mathcal{G}$ be a $1$-based stable group. Then every definable set $X\subseteq G^n$ is a Boolean combination
of cosets of almost $\emptyset$-definable subgroups of $G^n$. Moreover $G$ is abelian-by-finite.
\end{thm}

On the other hand, Hrushovski's seminal work in refuting Zilber's trichotomy conjecture (see \cite{HrushovskiSMSet})  
produced ``new'' strongly minimal sets that had an interesting property. Hrushovski isolated this property and 
called it $CM$-triviality (for Cohen-Macaulay). 
A stable theory $T$ is {\em $CM$-trivial} if there are no $a, b, c$ such that 
$a$ forks with $c$ over $\emptyset$, $a$ is independent from $c$ over $b$, $acl^{eq}(a)\cap acl^{eq}(b)=acl^{eq}(\emptyset)$ 
and finally $acl^{eq}(a,b)\cap acl^{eq}(a,c)=acl^{eq}(a)$. A kind of an analogue to the moreover statement of 
the above theorem has been proved by Pillay in \cite{PillFMR}.  

\begin{thm} 
A $CM$-trivial group of finite Morley rank is nilpotent-by-finite.
\end{thm}

Pillay first realized the pattern and proposed an hierarchy of ampleness, non $1$-basedness ($1$-ampleness) and non $CM$-triviality ($2$-ampleness) 
being the first two items in it (see \cite{PiAmp}). His definition needed some fine ``tuning'' as observed by Evans \cite{EvAmp}. 

\begin{defi}[\cite{EvAmp}]\label{Ample}
Let $T$ be a stable theory and $n\geq1$. Then $T$ is $n$-ample if (after possibly adding some parameters) there are 
$a_0,a_1,\ldots,a_n$ such that:
\begin{enumerate}
\item $a_0$ forks with $a_n$ over $\emptyset$;
\item $a_{i+1}$ does not fork with $a_0,\ldots,a_{i-1}$ over $a_i$, for $1\leq i<n$;
\item $acl^{eq}(a_0)\cap acl^{eq}(a_1)=acl^{eq}(\emptyset)$;
\item $acl^{eq}(a_0,\ldots,a_{i-1},a_i)\cap acl^{eq}(a_0,\ldots,a_{i-1},a_{i+1})=acl^{eq}(a_0,\ldots,a_{i-1})$, for $1\leq i<n$.
\end{enumerate}

\end{defi}

For the precise definitions of the above mentioned model theoretic notions we refer the reader to Section \ref{2}.

In this paper we give a sequence of primitive elements of $\F_{\omega}$ witnessing $n$-ampleness, for any $n<\omega$.
 
The paper is structured as follows. The next section serves as an introduction to the notions (of both model theory and geometric group theory) 
needed in order to place the result into context. The main purpose is to make our exposition as friendly as possible to 
the general reader, we still give references when we feel that this is not possible. 

In the third section we give a sequence of primitive elements and prove that this 
sequence witnesses $n$-ampleness for any $n<\omega$. Actually, our methods provide 
an ``abundance'' of examples to ampleness.

Finally, we add an appendix in which we record some strengthenings and alternative transparent proofs of  
Theorems 2.1, 2.2, and 2.3 from \cite{SelIm}, 
that prove that the ``basic'' imaginaries (see Definition \ref{Imaginaries}) cannot be eliminated. 



\begin{rmk}
The idea for the sequence witnessing ampleness (see Section \ref{3}) came to us after reading \cite{OuldTentAmple}, in which 
the main result is the ampleness of the theory of non-Abelian free groups (or more generally the theory of any (non-cyclic) 
torsion-free hyperbolic group). We first posted our sequence in \cite{SklinosAmple}, where we used results from \cite{OuldTentAmple} 
in order to prove that our sequence satisfies the algebraic criteria (3) and (4) of Definition \ref{Ample}. 

Conceptually the two sequences are much different and the main advantage of our sequence is that it 
reduces the work of satisfying the algebraic conditions to a (well-absorbed) fact about homeomorphisms of surfaces. 
Thus, in this paper we give an alternative proof to the one given in \cite{SklinosAmple} with the hope that it will add to 
the understanding of the first-order theory of non-Abelian free groups.
\end{rmk}

\paragraph{Acknowledgements} We would like to thank Zlil Sela for some helpful discussions and Gilbert Levitt for a useful correspondence. Many thanks to the anonymous referee whose comments improved the exposition of the article.

\section{Preliminaries}\label{2}
In this section we collect some basic definitions and facts about model theory and geometric group theory. To be more precise 
in the first subsection we will define the notion of imaginaries and explain various notions of elimination of imaginaries. 
In the second subsection we explain the notion of forking independence and we connect it with the notion of 
generic types in stable groups. In the next section, we 
specialize these notions to the first order theory of the free group. In the 
fourth subsection we will define amalgamated free products and give some normal form theorems for elements in an amalgamated free product. 
In the last subsection we will define pseudo-Anosov homeomorphisms of surfaces and record some useful facts about them. 

The reader should note that our treatment is by no means complete but we will always provide references for notions and results 
that are not adequately explained.

\subsection{Imaginaries}
We fix a first order structure $\mathcal{M}$ and we are interested in the collection of definable sets in $\mathcal{M}$, i.e. all subsets 
of some cartesian power of $\mathcal{M}$ which are the solution sets of first order formulas (in $\mathcal{M}$). 
In some cases one can easily describe this collection usually thanks to some quantifier elimination result. 
For example, as algebraically closed fields admit (full) quantifier elimination (in the language of rings) 
the class of definable sets coincides with the class of constructible sets, i.e. the class consisting of boolean combinations of Zariski closed sets. 
On the other hand, although free groups admit quantifier elimination
down to boolean combinations of $\forall\exists$ formulas (see \cite{Sel5,Sel5bis}), the ``basic'' definable sets are not so easy to describe. 
 
Suppose $X$ is a definable set in $\mathcal{M}$. One might ask whether there is a canonical way to 
define $X$, i.e. is there a tuple $\bar{b}$ and a formula $\psi(\bar{x},\bar{y})$ such that $\psi(\mathcal{M},\bar{b})=X$ 
but for any other $\bar{b}'\neq \bar{b}$, $\psi(\mathcal{M},\bar{b}')\neq X$? 

To give a positive answer to the above mentioned question one has to move to a ``mild'' expansion of 
$\mathcal{M}$ called $\mathcal{M}^{eq}$. Very briefly $\mathcal{M}^{eq}$ is constructed from $\mathcal{M}$ 
by adding a new sort for each $\emptyset$-definable equivalence relation, $E(\bar{x},\bar{y})$, together 
with a class function $f_E:M^n\rightarrow M_E$, where $M_E$ (the domain of the new sort corresponding to $E$) 
is the set of all $E$-equivalence classes. The elements in these new sorts are called {\em imaginaries}. 
In $\mathcal{M}^{eq}$, it is not hard to see that one can assign to each definable set a canonical parameter in the sense discussed above. 
Indeed, let $X$ be the solution set of the formula $\phi(\bar{x},\bar{b})$ in $\mathcal{M}$ and consider the equivalence relation 
$E(\bar{y}_1,\bar{y}_2):=\forall\bar{x}(\phi(\bar{x},\bar{y}_1)\leftrightarrow \phi(\bar{x},\bar{y}_2))$. Then 
the element $f_E(\bar{b})$ in $\mathcal{M}^{eq}$ serves as the ``canonical parameter'' that when it is plugged in 
in the formula $\psi(\bar{x},z_E):=\exists \bar{y} (\phi(\bar{x},\bar{y})\land f_E(\bar{y})=z_E)$, where $z_E$ denotes a 
variable that takes values in the $E$-sort, defines canonically the set $X$.

An element $a$ of $\mathcal{M}^{eq}$ is {\em algebraic} (respectively {\em definable}) over $A\subseteq \mathcal{M}^{eq}$, 
denoted $a\in acl^{eq}(A)$ (respectively $a\in dcl^{eq}(A)$), 
if there exists a first order formula over $A$ with finitely many solutions 
(respectively exactly one solution) in $\mathcal{M}^{eq}$ containing $a$. 

We say that $\mathcal{M}$ {\em eliminates imaginaries} if there is a saturated elementary extension 
of $\mathcal{M}$ in which all definable sets can be assigned a canonical parameter. Equivalently, 
$\mathcal{M}$ eliminates imaginaries if it has a saturated elementary extension $\mathbb{M}$ 
in which for any element $\mathfrak{e}$ of $\mathbb{M}^{eq}$, 
there is a finite tuple $\bar{b}\in\mathbb{M}$ such that $\mathfrak{e}\in dcl^{eq}(\bar{b})$ and $\bar{b}\in dcl^{eq}(\mathfrak{e})$.

One can alter the above definition to obtain the following weaker notions of elimination of imaginaries.  
We say that $\mathcal{M}$ {\em weakly} (respectively {\em geometrically}) {\em eliminates imaginaries} 
if it has a saturated elementary extension $\mathbb{M}$ 
in which for any element $\mathfrak{e}$ of $\mathbb{M}^{eq}$, 
there is a finite tuple $\bar{b}\in\mathbb{M}$ such that $\mathfrak{e}\in dcl^{eq}(\bar{b})$ (respectively $\mathfrak{e}\in acl^{eq}(\bar{b})$) 
and $\bar{b}\in acl^{eq}(\mathfrak{e})$. 

The interested reader can find more details in \cite[Sections 16.4 \& 16.5]{PoizatModelTheory}.

 


\subsection{Forking independence}
For a quick introduction to forking independence in stable theories we refer the reader to \cite[Section 2]{LouderPerinSklinosTowers} 
or \cite[Section 2.2]{PerinSklinosForking}, more thorough references are \cite{PillayStability} 
and \cite{PoizatModelTheory}. 

A first order theory $T$ is called stable if it supports a notion of independence (between tuples in an ``enough'' saturated model of $T$) 
satisfying certain properties. As a matter of fact, in a stable theory, there is exactly one notion of independence with the 
desired properties, which is called forking independence. 

We work in a ``big'' saturated model, $\mathbb{M}$, of a stable theory $T$ (what model theorists oftenly call the {\em monster model} 
see \cite[p.218]{MarkerModelTheory}). 
Let $\bar{a},\bar{b},\ldots$ denote finite tuples in $\mathbb{M}$ and $A,B,\ldots$ 
small subset of $\mathbb{M}$, i.e. $\abs{A},\abs{B}<\abs{\mathbb{M}}$. We say that $\bar{a}$ {\em forks with} (is {\em not independent from}) $B$ over $A$ if there 
is $\phi(\bar{x},\bar{b})\in tp(\bar{a} / B)$ and an indiscernible sequence $(\bar{b}_i)_{i<\omega}$ over $A$ with $tp(\bar{b} / A)=tp(\bar{b}_i / A)$, 
such that $\{\phi(\bar{x},\bar{b}_i) \ | \ i<\omega\}$ is inconsistent. 

Moreover a sequence of tuples $\bar{a}_1,\ldots, \bar{a}_k$ is called an {\em independent set (over A)} if $\bar{a}_i$ does not 
fork with $\bar{a}_1,\ldots,\bar{a}_{i-1}$ over $A$, for all $i\leq k$.

A group, $G$, is called stable if it is definable in some model of a stable theory $T$. An important aspect of stable groups is the 
existence of generic types. A type $tp(a/A)$ with $a\in G$ is called {\em generic} if 
whenever $g\in G$ and $a$ does not fork with $g$ over $A$, then $g\cdot a$ does not fork with $g$ over $A$. 
Moreover, an element $a\in G$ is called {\em generic} (over $A$) if $tp(a / A)$ is generic. 

The development of 
stable group theory in full generality is mostly due to B. Poizat and an elegant reference for a more thorough reading 
as well as for motivating the above definitions is \cite{PoizatStableGroups}.


\subsection{The free group}
We now specialize all the above model theoretic notions 
to the first order theory of the free group (considered in the natural language for groups, i.e. $(\cdot,^{-1},1)$).

We start by defining some ``basic'' families of imaginaries. 

\begin{defi}\label{Imaginaries}
Let $\F$ be a non abelian free group. The following equivalence relations in $\F$ are called basic.
\begin{itemize}
 \item[(i)] $E_1(a,b)$ if and only if there is $g\in \F$ such that $a^g=b$. (conjugation)
 \item[$(ii)_m$] $E_{2_m}((a_1,b_1),(a_2,b_2))$ if and only if either $b_1=b_2=1$ or $b_1\neq 1$ and $C_{\F}(b_1)=C_{\F}(b_2)=\langle b \rangle$ and
$a_1^{-1}a_2\in\langle b^m \rangle$. ($m$-left-coset)
 \item[$(iii)_m$] $E_{3_m}((a_1,b_1),(a_2,b_2))$ if and only if either $b_1=b_2=1$ or $b_1\neq 1$ and $C_{\F}(b_1)=C_{\F}(b_2)=\langle b \rangle$ and
$a_1a_2^{-1}\in\langle b^m \rangle$. ($m$-right-coset)
 \item[$(iv)_{m,n}$] $E_{4_{m,n}}((a_1,b_1,c_1),(a_2,b_2,c_2))$ if and only if either 
$a_1=a_2=1$ or $c_1=c_2=1$ or $a_1,c_1\neq 1$ and $C_{\F}(a_1)=C_{\F}(a_2)=\langle a \rangle$ and $C_{\F}(c_1)=C_{\F}(c_2)=\langle c \rangle$
  and there is $\gamma\in \langle a^m \rangle$ and $\epsilon\in \langle c^n \rangle$ such that $\gamma b_1 \epsilon=b_2$. ($m,n$-double-coset)
\end{itemize}

\end{defi}

It is almost immediate that $m$-left cosets eliminate $m$-right cosets (and vice versa), 
so from now on we are economic and forget about the $m$-right-cosets.

Sela proved the following theorem concerning imaginaries in non abelian free groups (see \cite[Theorem 4.4]{SelIm}).

\begin{thm}\label{Elim}
Let $\F$ be a non abelian free group. Let $E(\bar{x},\bar{y})$ be a definable equivalence relation in $\F$, with $\abs{\bar{x}}=m$.
Then there exist $k,l<\omega$ and a definable relation:
$$R_E \subseteq \F^m \times \F^k \times S_1(\F) \times \ldots \times S_l(\F)$$
such that:
\begin{itemize}
 \item[(i)] each $S_i(\F)$ is one of the basic sorts;
 \item[(ii)] for each $\bar{a}\in \F^m$ , $\abs{R_E(\bar{a},\bar{z})}$ is uniformly bounded (i.e. the bound does not depend on $\bar{a}$);
 \item[(iii)] $\forall\bar{z}(R_E(\bar{a},\bar{z})\leftrightarrow R_E(\bar{b},\bar{z}))$ if and only if $E(\bar{a},\bar{b})$.
\end{itemize}
\end{thm}

If we denote by $\F^{we}:=(\F,S_1(\F),\{S_{2_m}(\F)\}_{m<\omega},\{S_{4_{m,n}}(\F)\}_{m,n<\omega})$, then the above theorem together with 
\cite[Proposition 4.5]{SelIm} implies:

\begin{thm}\label{WeakElim}
Let $\F$ be a non abelian free group. Then $\F^{we}$ weakly eliminates imaginaries.
\end{thm}

\begin{rmk}
One has to be careful with stating a $\emptyset$-definable (i.e. definable by a first-order formula without any parameters) 
version of Theorem \ref{Elim}. Actually it is easy to find a 
counterexample if one replaces definable by $\emptyset$-definable everywhere in the above theorem: 
Let $E$ be a $\emptyset$-definable equivalence relation with finitely many classes, then by \cite[Theorem 3.1]{PerinSklinosForking} 
each class is $\emptyset$-definable, but then the above relation can only assign to each class the single tuple consisting of trivial (imaginary) elements, i.e. $[(1,1,1)]_{E_{4_{m,n}}}$ or $[(1,1)]_{E_{2_{n}}}$ or $[1]_{E_1}$ or a trivial real element, a contradiction since we want the relation to distinguish between classes. 
\end{rmk}

Not long after the positive solution to Tarski's question (see \cite{Sel6,KharlampovichMyasnikov}), 
that is whether non abelian free groups share the same common theory, Sela proved 
the following astonishing result \cite{SelaStability}.

\begin{thm}
The first order theory of non abelian free groups, $T_{fg}$, is stable.
\end{thm}

We note that, by work of Poizat \cite{PoizatGenericAndRegular}, $T_{fg}$ is connected, i.e. there is no definable proper subgroup 
of finite index (in any model of $T_{fg}$). In stable theories this is equivalent to saying that there is a unique 
generic type over any set of parameters.

We now recall some results about forking independence in the theory of the free group. 
For the purpose of this paper the following theorems of Pillay concerning forking independence 
and generic elements are enough. We denote by $\F_n$ the free group of rank $n$.
 
\begin{thm}[Corollary 2.7(ii)\cite{PillayForking}]\label{ForkPil}
Let $n>1$. For any basis, $a_1,\ldots,a_n$, of $\F_n$ we have that $a_1,\ldots, a_n$ is an independent set of realizations 
of the (unique) generic type.
\end{thm}

Recall that an element of a free group $\F$ is called {\em primitive} if it is part of some basis of $\F$. 

\begin{thm}[Theorem 2.1(i)\cite{PillayGenericity}]\label{GenPil}
Let $n>1$. Suppose $a$ is a generic element in $\F_n$. Then $a$ is primitive.
\end{thm}

\subsection{Amalgamated Free Products}
In this subsection we recall some well known facts about amalgamated free products, 
we refer the reader to \cite[Chapter IV]{LyndonSchupp} or to \cite[Section 4.4]{MagnusKarrassSolitar} for more details and 
motivation. We fix two groups $A,B$ a subgroup $C$ of $A$ and an embedding $f:C\to B$. Then the {\em amalgamated free product} 
$G:=A*_CB$ is the group $\langle A,B | c=f(c), \  c\in C\rangle$. Note that $G$ can be viewed as the free product 
$A*B$ quotiened by the normal subgroup generated by $\{cf(c)^{-1} \ | \ c\in C\}$. This construction 
naturally arises in the context of algebraic topology for example in the Seifert - van Kampen theorem 
(see \cite[Section 1.2]{HatcherAlgTop}).

\begin{defi}[Reduced forms]
A product of elements $g_1\cdot\ldots\cdot g_n$ from $A*B$ for $n\geq 0$ 
is in reduced form if the following conditions hold:
\begin{itemize}
 \item for each $i\leq n$, $g_i\in A\cup B$ and $g_i,g_{i+1}$ belong to different factors;
 \item if $n>1$, then no $g_i$ belongs to $C$ or $f(C)$;
 \item if $n=1$, then $g_1\neq 1$.
\end{itemize}
\end{defi}

Clearly, any element $g\in G$ can be written as a product of elements in reduced form, but this form is not unique. 

We can obtain uniqueness once we fix systems of representatives for the right cosets of $C$ in $A$ and 
for the right cosets of $f(C)$ in $B$. 

\begin{defi}[Normal forms]
Let $S$ (respectively $T$) be a system of right coset representatives for $C$ in $A$ 
(respectively a system of right coset representatives for $f(C)$ in $B$). Then a product 
of elements $c\cdot g_1\cdot\ldots\cdot g_n$ from $A*B$ is in normal form if $c\in C$ and 
$g_1\cdot\ldots\cdot g_n$ is in reduced form with each $g_i$ belonging to $S\cup T$.
\end{defi}

Then, we have:

\begin{thm}[Normal Form Theorem]\label{NormForm}
Let $S$ (respectively $T$) be a system of right coset representatives for $C$ in $A$ 
(respectively a system of right coset representatives for $f(C)$ in $B$). Let $g\in G$. 
Then $g$ can be uniquely represented as a product of elements in normal form. 
\end{thm}

Since we will do many calculations with normal forms, we give more details for a situation that will often occur. Fix $S$ (respectively $T$)  
a system of right coset representatives for $C$ in $A$ (respectively $f(C)$ in $B$). If $g\in A$ (respectively $g\in B$), 
denote by $\hat{g}$ (respectively $\tilde{g}$) the element 
in $S$ (respectively in $T$) such that $Cg=C\hat{g}$ (respectively $f(C)g=f(C)\tilde{g}$). Let 
$\gamma=cg_1g_2\ldots g_n$ be an element in normal form and let $a$ be an element in $A$. We would like to calculate 
the normal form of $\gamma\cdot a$. We take cases with respect to whether $g_n$ is in $A$ or $B$:

\begin{itemize}
 \item suppose that $g_n$ is in $A$. Then the normal form of $\gamma\cdot a$ is 
 $c\cdot c_1 \widehat{g_1c_2}\widetilde{g_2c_3}\ldots\widetilde{g_{n-1}c_n}$ $\widehat{g_na}$, where the $c_i$'s belong to $C$ 
 and $g_na=c_n\widehat{g_na}$, $g_{n-1}c_n=c_{n-1}\widetilde{g_{n-1}c_n}$, $\ldots$, $g_1c_2=c_1\widehat{g_1c_2}$.
 \item suppose that $g_n$ is in $B$. In this case, if $a$ is in $C$, then $\gamma\cdot a$ is 
 $c\cdot c_1 \widehat{g_1c_2}\widetilde{g_2c_3}\ldots\widehat{g_{n-1}c_n}$ $\widetilde{g_na}$, where the $c_i$'s belong to $C$ 
 and $g_na=c_n\widetilde{g_na}$, $g_{n-1}c_n=c_{n-1}\widehat{g_{n-1}c_n}$, $\ldots$, $g_1c_2=c_1\widehat{g_1c_2}$.
 
 If $a$ is not in $C$, then $\gamma \cdot a$ has the following normal form $c\cdot c_1 \widehat{g_1c_2}\widetilde{g_2c_3}
 \ldots\widehat{g_{n-1}c_n}$ $\widetilde{g_nc_{n+1}}\widehat{a}$, with $a=c_{n+1}\widehat{a}$ and the obvious equations for the rest. 
\end{itemize}
 
A product in reduced form, $g_1\cdot\ldots\cdot g_n$, is called {\em cyclically reduced}, 
if any cyclic permutation of the $g_i$'s gives a product in reduced form. Equivalently $g_1\cdot\ldots\cdot g_n$ 
is cyclically reduced if $n=1$ or $n$ is even. We moreover have:

\begin{thm}[Conjugacy Theorem for Amalgamated Free Products]\label{NormConjForm}
Every element of $G$ is conjugate to an element that can be represented as a product in a cyclically reduced form. 

Moreover, if $g:=g_1\cdot\ldots\cdot g_n, h:=h_1\cdot\ldots\cdot h_m$ are products in cyclically reduced form, which are conjugates in $G$. 
Then:
\begin{itemize}
 \item[(i)] if $n=1$ and $g\in (A\cup B)\setminus C$, 
 then $m=1$, $h$ belongs to the same factor as $g$ and they are conjugates by an element of this factor;
 \item[(ii)] if $n=1$ and $g\in C$, then $m=1$ and there is a sequence of elements $g,g_1,\ldots, g_l, h$ 
 where $g_i\in C$ and consecutive elements in the sequence are conjugates in a factor;
 \item[(iii)] if $n>1$, then $n=m$ and 
 $h$ can be obtained from $g$ by a cyclic permutation of $g_1,\ldots, g_n$ and then conjugation by an element of $C$. 
\end{itemize} 
\end{thm}

\subsection{Pseudo-Anosov homeomorphisms}
A homeomorphism, $h$, of a (compact) surface $\Sigma$ is called {\em pseudo-Anosov} if there exist a pair of transverse 
measured foliations, $(\mathcal{F}^u,\mu_u), (\mathcal{F}^s,\mu_s)$ and a real number $\lambda>1$, such that 
$h$ ``respects'' the foliations in the following sense: 
$$h.(\mathcal{F}^u,\mu_u)=(\mathcal{F}^u,\lambda\cdot\mu_u)\ \textrm{and} \ h.(\mathcal{F}^s,\mu_s)=(\mathcal{F}^s,\lambda^{-1}\cdot\mu_s)$$ 

The (isotopy) classes of pseudo-Anosov homeomorphisms play an important role in the study of the mapping class group $\mathcal{MCG}(\Sigma)$ of a 
(compact) surface $\Sigma$, i.e. the group of isotopy classes of orientation preserving homeomorphisms of $\Sigma$ (fixing the boundary components pointwise). 
Let us also note that examples of pseudo-Anosov homeomorphisms have been first consider by Nielsen (see \cite{NielsenI,NielsenII,NielsenIII}) but 
more sytematically studied after the work of Thurston (see \cite{ThurstonPseudo}), where he stated the following celebrated theorem.

\begin{thm}[Nielsen-Thurston classification theorem]
Let $\Sigma$ be a (compact) surface. Let $h\in\mathcal{MCG}(\Sigma)$. Then $h$ is either periodic, or reducible or 
pseudo-Anosov.
\end{thm}

For motivating the definition of a pseudo-Anosov one might consider the case of an Anosov homeomorphisms of the torus. 
We identify the torus with $\R^2/\Z^2$. The mapping class group of the torus is isomorphic to $SL_2(\Z)$ and an Anosov homeomorphism 
would be a matrix $A\in SL_2(\Z)$ with $\abs{trace(A)}>2$. For such a matrix we have two real eigenvalues, $\lambda>1$ and $\lambda^{-1}$, and 
the corresponding eigenlines in $\R^2$ have irrational slope. Moreover, one of the eigenlines is ``streched'' by a factor of $\lambda$ 
while the other is ``contracted'' by a factor of $\lambda^{-1}$. For each eigenline, the lines parallel to it form a foliation of $\R^2$ and 
the two foliations corresponding to the distinct eigenlines are transverse at each point. 
Since $\Z^2$ acts on the set of parallel lines, the foliations project to foliations of the torus, where each ``leaf'', i.e. the image 
of a line, is dense in $T^2$ and $A$ leaves each of the foliations invariant. For a more thorough exposition of the above notions
and results we refer the reader to \cite[Chapter 13]{Primer} or \cite{ThurstonSurface}.

We now collect some useful properties of pseudo-Anosov homeomorphisms, we believe well known, in the following theorem. 
We still sketch a proof which will be rather quick and hard to follow for the reader lacking geometric background. 

\begin{thm}\label{Pseudo}
Let $\Sigma_{g,1}$ be the orientable surface of genus $g$ with connected (non-empty) boundary component. 
Let $\pi_1(\Sigma_{g,1},*)$ be the fundamental group of $\Sigma_{g,1}$ with respect to the base point $*$, and 
let $B$ be a maximal boundary subgroup.

Suppose $h$ is a pseudo-Anosov homeomorphism of $\Sigma_{g,1}$ 
and $[h_*]$ is the corresponding outer automorphism of $\pi_1(\Sigma_{g,1},*)$. 
Then:
\begin{itemize}
 \item[(i)] if $a\in\pi_1(\Sigma_{g,1},*)$ cannot be conjugated to an element in the boundary subgroup $B$, then 
 $\{[h_*]^k.[a] \ | \ k\in\omega\}$ is infinite, where $[a]$ denotes the conjugacy class of $a$;
 \item[(ii)] if $h_*\in [h_*]$ is an automorphism of $\pi_1(\Sigma_{g,1},*)$ that fixes the boundary subgroup $B$, then 
 the orbit of double cosets $B.a.B$ under powers of $h_*$, 
 $\{B.h_*^k(a).B \ | \ k\in\omega\}$, is infinite for any $a\not\in B$.
\end{itemize}
Moreover, both (i) and (ii) hold for any infinite subsequence of powers of $[h_*]$. 
\end{thm}
\begin{proof}[Sketch of proof]

For both parts of the theorem we will use the following fact: Any element, $a\in\pi_1(\Sigma_{g,1})$, 
that cannot be conjugated to an element in $B$ 
has uniform exponential growth under powers of $h_*$, i.e. there exists a ``stretching factor'' $\lambda_{h_*}>1$ 
such that:
$$\abs{\widehat{h_*^k(a)}}_{\F_{2g}}\sim C_a\cdot\lambda^k_{h_*}$$ 
where $C_a$ is a constant depending only on the element $a$ (and the choice of the generating set for $\F_{2g}$), 
and $\hat{a}$ denotes the cyclically reduced element (up to cyclic permutation) in the conjugacy class of $a$. 

Part $(i)$ follows immediately from the above fact. 


For $(ii)$, we consider the action of $\pi_1(\Sigma_{g,1})$ on a based real tree $(T,x)$ obtained by 
the Bestvina-Paulin method (see \cite{BesDeg,PaulinGromov}) 
from the sequence of automorphisms $(h_*^k)_{k\in\omega}:\pi_1(\Sigma_{g,1})\to\F_{2g}$ (or any infinite subsequence). Using 
the above fact one can easily verify the following properties of the limiting action. First, 
if an element $a\in\pi_1(\Sigma_{g,1})$ cannot be conjugated to an element in $B$, then $a$ acts hyperbolically on $T$, moreover 
the translation length of $h_*^k(a)$, $tr_T(h_*^k(a))$, goes to infinity, as $k\to\infty$. Second, any non trivial 
element of $B$ fixes exactly $x$.

Now, suppose for the sake of contradiction, that for some $a\not\in B$ we have that $B.h_*^k(a).B$ $=B.a.B$ for 
arbitrarily large $k$. Then, we clearly have that  
$tr_T(h_*^k(a))=tr_T(b^{n_k}ab^{m_k})=tr_T(ab^{m_k-n_k})$ and $\abs{m_k-n_k}\to\infty$, as $k\to\infty$. We take cases: 
\begin{itemize}
 \item[Case 1] Suppose $a$ can be conjugated to an element in $B$, then $a$ fixes a point in $T$ which is 
 different from $x$. Thus, $tr_T(a\cdot b^{m_k-n_k})>0$, but $tr_T(h_*^k(a))=tr(a)=0$, a contradiction;
 \item[Case 2] Suppose $a$ cannot be conjugated to an element in $B$. Then, $tr_T(h_*^k(a))\to\infty$ as $k\to\infty$, 
 but $tr_T(ab^{m_k-n_k})$ is bounded by $d(x,a\cdot x)$, a contradiction.
\end{itemize}
\end{proof}

The same is true for non-orientable surfaces.

\begin{thm}\label{PseudoNonOrien}
Let $\Pi_{n,1}$ be the connected sum of $n$ projective planes with connected (non-empty) boundary component. 
Let $\pi_1(\Pi_{n,1},*)$ be the fundamental group of $\Pi_{n,1}$ with respect to the base point $*$, and 
let $B$ be a maximal boundary subgroup.

Suppose $h$ is a pseudo-Anosov homeomorphism of $\Pi_{n,1}$ 
and $[h_*]$ is the corresponding outer automorphism of $\pi_1(\Pi_{n,1},*)$. 
Then:
\begin{itemize}
 \item[(i)] if $a\in\pi_1(\Pi_{n,1},*)$ cannot be conjugated to an element in the boundary subgroup $B$, then 
 $\{[h_*]^k.[a] \ | \ k\in\omega\}$ is infinite, where $[a]$ denotes the conjugacy class of $a$;
 \item[(ii)] if $h_*\in [h_*]$ is an automorphism of $\pi_1(\Pi_{n,1},*)$ that fixes the boundary subgroup $B$, then 
 the orbit of double cosets $B.a.B$ under powers of $h_*$, 
 $\{B.h_*^k(a).B \ | \ k\in\omega\}$, is infinite for any $a\not\in B$.
\end{itemize}
Moreover, both (i) and (ii) hold for any infinite subsequence of powers of $[h_*]$. 
\end{thm}

We note that most surfaces support pseudo-Anosov homeomorphisms.

\begin{fact}[cf. \cite{PennerPseudo}]
Let $\Sigma$ be either the torus with connected boundary or a (possibly non-orientable) surface with Euler characteristic at most $-2$, 
then it carries a pseudo-Anosov homeomorphism. 
\end{fact}

\section{Witnessing Ampleness}\label{3}
In this section we prove the main result of the paper. We will show that the following sequence in 
$\F_{\omega}:=\langle e_1,e_2,\ldots,e_k,\ldots\rangle$ witnesses $n$-ampleness, 
for any $n\in\omega$ (after adding $e_1,e_2$ as parameters). We give the sequence recursively:

$$a_0=e_3$$ 
$$a_{i+1}=a_i[e_{2i+4},e_{2i+5}], \ \textrm{for} \ i\in\omega$$

We fix a natural number $n\geq 1$, and we show that $a_0,\ldots, a_n$ witnesses $n$-ampleness 
by verifying the requirements of Definition \ref{Ample}. 

We can now proceed with the proofs of the first three requirements of Definition \ref{Ample}.

\begin{prop}\label{D1}
$a_0=e_3$ forks with $a_n=e_3[e_4,e_5]\ldots[e_{2n+2},e_{2n+3}]$ over $e_1,e_2$.
\end{prop}

\begin{proof}
Suppose not. Note that since $a_n$ together with $e_1,e_2$ can be completed to form a basis of $\F_{2n+3}$, we have that 
$a_n$ is generic over $e_1,e_2$. But then, by our contradictory hypothesis, $a_n$ is generic over $e_1,e_2,e_3$. Consequently, 
$e_3^{-1}a_n=[e_4,e_5]$ $\ldots[e_{2n+2},e_{2n+3}]$ is generic over  
$e_1,e_2, e_3$ and, by Theorem \ref{GenPil}, $[e_4,e_5]\ldots$ $[e_{2n+2},e_{2n+3}]$ is a primitive element. This is 
a contradiction since a primitive element of $\F_n$ maps to a primitive element in the abelianization $\Z^n=\F_n/[\F_n,\F_n]$.
\end{proof}

\begin{prop}\label{D2}
Let $1\leq i<n$. Then $a_0,\ldots,a_{i-1}$ does not fork with $a_{i+1}$ over $e_1,e_2,a_i$. 
\end{prop}

\begin{proof}
We first note that for each $i$, $\langle e_1,e_2, a_i\rangle$ is a free factor of $\F_{2i+5}$, i.e. $e_1,e_2,a_i$ extends to a basis of $\F_{2i+5}$. 
Thus, by Theorem \ref{ForkPil}, we only need to find a 
free factorization $\F_{2i+5}=\F*\langle e_1,e_2, a_i\rangle* \F'$, such that $a_0,\ldots,a_{i-1}$ is in 
$\F*\langle e_1,e_2,a_i\rangle$ and $a_{i+1}$ is in $\langle e_1,e_2,a_i\rangle* \F'$. It is easy to see that the following free factorization is such: 
$$\F_{2i+5}=\langle e_4,e_5, \ldots, e_{2i+2},e_{2i+3}\rangle *\langle e_1,e_2, a_i\rangle * \langle e_{2i+4},e_{2i+5}\rangle$$
\end{proof}

\begin{prop}\label{D3}
$\F_3^{eq}\cap \langle e_1,e_2,e_3[e_4,e_5]\rangle^{eq}=\F_2^{eq}$. 
 \end{prop}

\begin{proof}
Let $A:=\langle e_1,e_2,e_3[e_4,e_5]\rangle$ and $\alpha \in \F_3^{eq}\cap A^{eq}$.  
Then there exists an equivalence relation $E$ and a tuple $\bar{a}$ consisting of elements of $\F_5$ such that $\alpha=[\bar{a}]_E$.  
Thus, by Theorem \ref{Elim}, we have that $R_E(\bar{a},\F_5^{eq})=\{\bar{\alpha}_1,\ldots,\bar{\alpha}_k\}$. Note 
that each $\bar{\alpha}_i$ is algebraic over $\F_2\alpha$, thus they all belong to $\F_3^{eq}\cap A^{eq}$. 
Now let $\beta$ be an element of the tuple $\bar{\alpha}_i$ for some $i\leq k$. We take cases for $\beta$. 
\begin{itemize}
 \item[(i)] Suppose $\beta\in \F_5$. Then $\beta\in \F_3\cap A$, which is exactly $\F_2$.
 \item[(ii)] Suppose $\beta=[b]_E$, with $E=E_1$. We may assume that $b$ is in a cyclically reduced form  
 with respect to the free splitting $\F_2*\langle e_3,e_4,e_5\rangle$. Since $\beta\in\F_3^{eq}$ (respectively $\beta\in A^{eq}$) 
 there is $b_1\in\F_3$ (respectively $b_2\in A$) such that $[b_1]_E=[b_2]_E=[b]_E$. But a cyclically reduced form for $b_1$ with respect 
 to $\F_2*\langle e_3\rangle$ (respectively for $b_2$ with respect to $\F_2*\langle e_3[e_4,e_5]\rangle$)  will automatically be 
 a cyclically reduced form with respect to $\F_2*\langle e_3,e_4,e_5\rangle$. Therefore, $b_1$ is a cyclic permutation of $b_2$ 
 which implies that $b\in \F_3\cap A$, as we wanted.
 \item[(iii)] Suppose $\beta=[(b_1,b_2)]_E$, with $E=E_{2_m}$. If $b_2$ is the identity element then the result holds trivially, thus 
 we may assume that $b_2\neq 1$. Since $\beta\in \F_3^{eq}$ (respectively $\beta\in A^{eq}$) 
 there is $b_{21}\in\F_3$ (respectively $b_{22}\in A$)  
 such that $C_{\F_5}(b_2)=C_{\F_5}(b_{21})=C_{\F_5}(b_{22})$. Therefore,  
 there are $b\in\F_5$ and $k,l,p\in\Z\setminus\{0\}$ such that $b^k=b_{21}$, $b^l=b_{22}$ and $b^p=b_2$. But, since 
 $\F_3$ (respectively $A$) is a free factor of $\F_5$, 
 if some power of an element of $\F_5$ belongs to $\F_3$ (respectively to $A$), then the element itself belongs to $\F_3$ 
 (respectively $A$), thus $b\in\F_3\cap A$ and consequently $b_2\in\F_2$.
 
 Moreover, there are $b_{11}\in \F_3$ and $b_{12}\in A$ such that $b_1.C(b_2)=b_{11}.C(b_2)=b_{12}.C(b_2)$, 
 thus $b_{11}b_{12}^{-1}\in \F_2$ and consequently all $b_1,b_{11},b_{12}$ belong to $\F_2$.
 \item[(iv)] Suppose $\beta=[(b_1,b_2,b_3)]_E$, with $E=E_{4_{m,n}}$. Then the proof that $b_1, b_2, b_3$ belong to $\F_2$ 
 is identical with the previous case.
\end{itemize}
This shows that $\alpha\in\F_2^{eq}$, as we wanted.
\end{proof}

In order to verify the fourth requirement of Definition \ref{Ample} we first need some preparatory lemmata. We formalize 
a construction that will often occur and point out some easy observations in the following remark:

\begin{rmk}\label{ExtPseudo}
\ 
\begin{itemize}
\item We first realize $\F_{2g}$ as the fundamental group, $\pi_1(\Sigma_{g,1})$, of the orientable surface of genus $g>0$ 
with connected boundary. Let $B$ be a fixed boundary subgroup of $\pi_1(\Sigma_{g,1})$. Suppose 
$A$ is a group and $f:B\rightarrow A$ is an injective morphism. 

We consider the amalgamated free product $G=\F_{2g}*_BA$ of $\F_{2g}$ with $A$ over $\{B,f\}$ and an automorphism $\alpha\in Aut_B(\F_{2g})$ coming 
from a pseudo-Anosov homeomorphism. Then $\alpha$ extends to an automorphism fixing (pointwise) $A$, and we will call such an automorphism an extension of a pseudo-Anosov homeomorphism. 

More generally, any representative (in the outer class) of an outer automorphism corresponding to a homeomorphism (fixing the boundary pointwise) can be extended to an automorphism of $G$, restricting to conjugation on $A$.

\item The subgroup $A:= \langle e_1,e_2,e_3,[e_4,e_5],\ldots,[e_{2i+4},e_{2i+5}]\rangle$ of $\F_{2i+5}$ is root closed, i.e. 
if, for some $a\in \F_{2i+5}$, $a^m\in A$, then $a\in A$. This is not hard to see either by using normal forms for free products or 
more elegantly by using Bass-Serre theory (for the basic notions of Bass-Serre
theory 
we refer the reader to \cite{SerreTrees}). Consider the 
action on a simplicial tree corresponding to the graph of groups decomposition of the left-side of Figure \ref{Fig1}. 


We observe that the edge stabilizers for the action are root closed in $\F_{2i+5}$: an edge stabilizer will be a conjugate of an edge group indicated on the left side of Figure \ref{Fig1}. We also note that an element of $\F_{2i+5}$ can either be elliptic or hyperbolic with respect to its action on the Bass-Serre tree. An element is elliptic if it fixes a point in the tree and hyperbolic if not, in the latter case the element admits an invariant axis (i.e. a subtree isometric to the real line) on which it acts by translations. If an element is hyperbolic, then any of its non-trivial powers is hyperbolic with the same axis.  

Now, since $a^m$ belongs to $A$ it fixes the vertex stabilized by $A$, call it $x$, and it is elliptic by definition. By our discussion above $a$ must also be elliptic. Assume, for a contradiction, that the element $a$ fixes a vertex different from $x$, call it $y$. Then $a^m$ fixes the segment between $x$ and $y$. Thus, $a^m$ fixes an edge adjacent to $x$. Since edge stabilizers are root closed, $a$ fixes the same edge, therefore it must fix $x$, a contradiction. 

\item The subgroup $\langle e_1,e_2,e_3,\ldots,[e_{2i+2},e_{2i+3}],
[e_{2i+4},e_{2i+5}][e_{2i+6},e_{2i+7}]\rangle$ of $\F_{2i+7}$ is root closed. The arguments of the 
previous point, considering the graph of groups decomposition on the left-side of Figure \ref{Fig2}, are also valid in this case.
\end{itemize}
\end{rmk}

\begin{figure}[ht!]
\centering
\includegraphics[width=.8\textwidth]{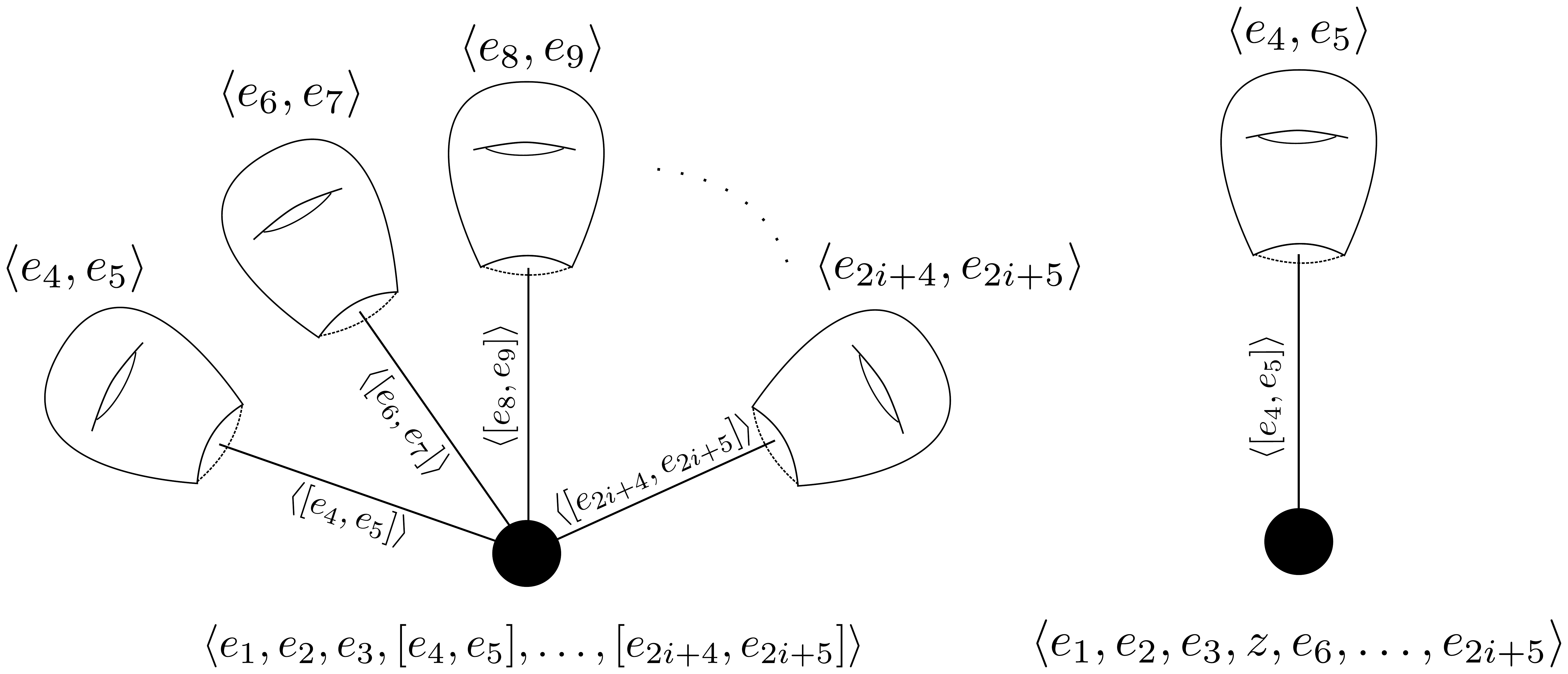}
\caption{A series of amalgamated free products (left-side). An amalgamated free product (right-side)}\label{Fig1}
\end{figure}

\begin{lemma}\label{BasicAlgebraicClosure2}
Suppose $\gamma\in acl^{eq}(e_1,e_2,a_0,a_1,\ldots,a_{i+1})\cap \F_{2i+5}^{we}$, for some $i<n$. Then:
\begin{itemize}
 \item if $\gamma$ is a real element, then $\gamma\in \langle e_1,e_2,$ $a_0,a_1,\ldots,a_{i+1}\rangle$;
 \item if $\gamma=[\bar{c}]_E$ for some basic equivalence relation $E$, then there is 
 $\bar{d}\in\langle e_1,e_2,a_0,a_1,$ $\ldots,a_{i+1}\rangle$ with  $\bar{c}\sim_E\bar{d}$.
\end{itemize} 
\end{lemma}
\begin{proof}
First we note that $\langle e_1,e_2,a_0,a_1,\ldots,a_{i+1}\rangle=\langle e_1,e_2,e_3,[e_4,e_5],\ldots,[e_{2i+4},e_{2i+5}]\rangle$ and we denote 
this group by $A$. 
We take cases for $\gamma$.
\begin{itemize}
 \item[(i)] Let $\gamma$ be an element of $\F_{2i+5}$. Suppose, for the sake of contradiction, 
 that $\gamma\not\in A$. We consider $\F_{2i+5}$ as the following finite sequence 
 of amalgamated free products (see figure \ref{Fig1}): 
 $$(((A*_{[e_{2i+4},e_{2i+5}]}\langle e_{2i+4},e_{2i+5}\rangle)*_{[e_{2i+2},e_{2i+3}]}\langle e_{2i+2},e_{2i+3}\rangle)\ldots)*_{[e_4,e_5]}\langle e_4,e_5\rangle$$
We denote by $A_1$ the amalgamated free product $A*_{[e_{2i+4},e_{2i+5}]}\langle e_{2i+4},e_{2i+5}\rangle$, and 
by $A_{l+1}$ the amalgamated free product $A_l*_{[e_{2i+4-2l},e_{2i+5-2l}]}\langle e_{2i+4-2l},e_{2i+5-2l}\rangle$ for $1\leq l\leq i$.\\
An easy induction shows that there is $l\leq i$, so that $\gamma$ admits a normal form, $\delta\gamma_1\gamma_2\ldots \gamma_m$, 
with respect to the amalgamated free product $A_l$, for which there is $\gamma_j$ for some $j\leq m$ that does not belong to $A_{l-1}$. 
Thus, without loss of generality we may assume that this is true for $l=i$ and 
$\gamma_j$ is an element of $A_i*_{\langle [e_4,e_5]\rangle}\langle e_4,e_5\rangle$ 
that does not belong to $A_i=\langle e_1,e_2,e_3,[e_4,e_5],e_6,\ldots,e_{2i+5}\rangle$.

 We will obtain a contradiction by showing that $\gamma$ has infinite orbit under $Aut_A(\F_{2i+5})$.

 We realize the group $\langle e_4,e_5\rangle$ 
 as the fundamental group of the torus with connected boundary and we fix $\langle [e_4,e_5]\rangle$ to be the preferred boundary subgroup. 
 Let $h_*\in Aut_A(\F_{2i+5})$ be an extension of a pseudo-Anosov homeomorphism. 
 By Theorem \ref{Pseudo}(ii), we have that for any 
 $a\in\langle e_4,e_5\rangle\setminus\langle [e_4,e_5]\rangle$, the left coset of the form    
 $a\cdot\langle [e_4,e_5]\rangle$ has an infinite orbit under powers of $h_*$. Since $\gamma$ is in $acl^{eq}(A)$ and 
 $h_*$ fixes $A$ we have that $\{h_*^k(\gamma) \ | \ k\in\omega\}$ is finite. Thus, for arbitrarily large $k$ 
 we get $h_*^k(\gamma)=\gamma$. Now, first assume that $\gamma_m\in\langle e_4,e_5\rangle$,  
 then $\delta\cdot h_*^k(\gamma_1)\ldots h_*^k(\gamma_m)\neq \delta\cdot\gamma_1\ldots\gamma_m$, since by our previous remark $h_*^k(\gamma_m)$ 
 would represent a different right coset from $\gamma_m$, contradicting Theorem \ref{NormForm}. In the case $\gamma_m\not\in\langle e_4,e_5\rangle$ 
 the same argument is valid for $\gamma_{m-1}$ which is necessarily in $\langle e_4,e_5\rangle$.
 
 \item[(ii)] Let $\gamma=[c]_E$, for $E=E_1$. Suppose, for the sake of contradiction, that $c$ cannot be 
 conjugated to an element in $A$. As in case (i), but using cyclically reduced forms now, we may assume that 
 $c$ admits a cyclically reduced form, $\gamma_1\gamma_2\ldots \gamma_m$, with respect to the amalgamated free product 
 $A_i*_{[e_4,e_5]}\langle e_4,e_5\rangle$ 
 so that for some $j\leq m$ we have that $\gamma_j$ does not belong to $A_i=\langle e_1,e_2,e_3,[e_4,e_5],e_6,\ldots,e_{2i+5}\rangle$.
 
 We will obtain a contradiction by showing that the conjugacy class of $c$ has infinite orbit under $Aut_A(\F_{2i+5})$. We 
 consider $h_*\in Aut_A(\F_{2i+5})$ to be an extension of a pseudo-Anosov homeomorphism obtained exactly as in Case (i). 
 As before, since $\{[h_*^k(c)] \ | \ k\in\omega\}$ is finite, $c$ is a conjugate of $h_*^k(c)$ for arbitrarily large $k$. 
 We now take cases for the length of the cyclically reduced form, 
 $\gamma_1\ldots \gamma_m$, for the element $c$. Note that we cannot have $c\in\langle[e_4,e_5]\rangle$.
    \begin{itemize}
    \item Suppose $m=1$ (thus $\gamma_1\in \langle e_4,e_5\rangle\setminus\langle[e_4,e_5]\rangle$), 
    and let $I$ be the infinite subset of $\omega$ for which for any $k\in I$, $h_*^k(\gamma_1)$ is a conjugate of $\gamma_1$. 
    Then, by Theorem \ref{NormConjForm}(i), $h_*^k(\gamma_1)$ and $\gamma_1$ are conjugates in $\langle e_4,e_5\rangle$, 
    but, by Theorem \ref{Pseudo}(i), $\{[h_*^k(\gamma_1)]_E \ | \ k\in I\}$ is infinite, a contradiction;
    \item Suppose $m>1$, and let 
    $h_*^k(\gamma_1\ldots \gamma_m)$ be a conjugate of $\gamma_1\ldots \gamma_m$ for arbitrarily large $k$. By Theorem \ref{NormConjForm}(iii) 
    we have that $h_*^k(\gamma_1)\ldots h_*^k(\gamma_m)$ is obtained from $\gamma_1\ldots \gamma_m$ by a cyclic permutation after possibly 
    conjugating by an element of the boundary subgroup. Thus, $b(k)^{-1}\gamma_{i_k(1)}\ldots \gamma_{i_{k}(m)}b(k)$ is $h_*^k(\gamma_1)\ldots h_*^k(\gamma_m)$ 
    for some cyclic permutation $i_k\in\langle(1 2 \ldots m)\rangle$ and some 
    $b(k)$ in $\langle[e_4,e_5]\rangle$,  
    for arbitrarily large $k$. Clearly this contradicts Theorem \ref{Pseudo}(ii).   
    \end{itemize} 
 
 \item[(iii)] Let $\gamma=[(c_1,c_2)]_{E_{2_p}}$. Recall that $\gamma$ is determined by the left coset $c_1\cdot C^p_{\F_{2i+5}}(c_2)$. 
 Suppose, for the sake of contradiction, that $c_2\not\in A$. Then as in 
 Case (i) we have that $c_2$ has infinite orbit under $Aut_A(\F_{2i+5})$. Let $(f_k)_{k\in\omega}\in Aut_A(\F_{2i+5})$ 
 be such that $f_k(c_2)\neq f_l(c_2)$ for $k\neq l$. Since $\gamma\in acl^{eq}(A)$ we have 
 that $(f_{k_1}(c_1),f_{k_1}(c_2))\sim_{E_{2_p}}(f_l(c_1),f_l(c_2))$ for some 
 $k_1\in\omega$ and arbitrarily large $l$. But then $C(f_{k_1}(c_2))=C(f_l(c_2))$ for arbitrarily large $l$, a contradiction 
 since for any automorphism $f\in Aut(\F_{2i+5})$ and any non-trivial element $c\in \F_{2i+5}$, if $C(c)=C(f(c))$
 then either $f(c)=c$ or $f(c)=c^{-1}$. Thus, $c_2\in A$. In the case that $c_2$ is the identity element we trivially have that 
 $c_1$ can be chosen to be in $A$, thus we may assume that $c_2\in A\setminus\{1\}$.
 
 Now assume, for the sake of contradiction, that $c_1\not\in A$ and as in Case (i) we may assume that $c_1$ admits a normal 
 form, $\delta_1\gamma_{11}\ldots\gamma_{1m}$, with respect to the amalgamated free product $A_i*_{[e_4,e_5]}\langle e_4,e_5\rangle$ 
 so that for some $j\leq m$ we have that $\gamma_{1j}$ does not belong to $A_i=\langle e_1,e_2,e_3,[e_4,e_5],e_6,\ldots,$ $e_{2i+5}\rangle$.
 We take $h_*$ to be an extension of a pseudo-Anosov homeomorphism exactly as in Case (i). 
 We will show that $\gamma$ has an infinite orbit under powers of $h_*$. 
 Suppose not, then the set of left cosets $\{h_*^k(c_1)\cdot C_{\F_{2i+5}}(c_2) \ | \ k\in\omega\}$ is finite. Thus, 
 for arbitrarily large $k$, $h_*^k(c_1)=c_1\cdot d^{n_k}$ for some $d$ in $A$
(recall that $c_2$ must be in $A$ and $A$ is root closed). 
 Consider the normal form, $\delta_1\beta_{11}\alpha_{11}\ldots\beta_{1m}\alpha_{1m}\beta_{1(m+1)}$, for $c_1$ where $\alpha_{1j}$ 
 belongs to $A_i$ (note that $\beta_{11}$ or $\beta_{1(m+1)}$ might be trivial). We take further cases: 
 \begin{itemize}
 
 \item Suppose $c_1$ is an element in $\langle e_4,e_5 \rangle \setminus \langle [e_4,e_5]\rangle$. First note that if $d$ 
 is not in the boundary subgroup $\langle [e_4,e_5]\rangle$, then $c_1d^{n_k}$ has different normal form from $h_*^k(c_1)$, thus we may assume that 
 $d$ is in the boundary subgroup. Now this easily contradicts Theorem \ref{Pseudo}(ii), since it implies that $\{B.h_*^k(c_1).B \ | \ k\in I\}$, for 
 some infinite subset $I$ of $\omega$, is finite.
 
 \item Suppose the last element in the normal form for $c_1$ is not in $A_i$ (i.e. $\beta_{1(m+1)}$ is not trivial). Then $d$ 
 must belong to the boundary subgroup $\langle[e_4,e_5]\rangle$, otherwise $\delta_1\beta_{11}\alpha_{11}\ldots\beta_{1m}\alpha_{1m}\beta_{1(m+1)}d^{n_k}$ 
 and $\delta_1h_*^k(\beta_{11})\alpha_{11}\ldots h_*^k(\beta_{1m})\alpha_{1m}$ $h_*^k(\beta_{1(m+1)})$ have different normal forms. Thus, 
 for $k$ arbitrarily large $\beta_{1(m+1)} d^{n_k}$ represent the same right coset with respect to the boundary subgroup as $h_*^k(\beta_{1(m+1)})$, 
 this contradicts Theorem \ref{Pseudo}(ii).
 
 \item Suppose the last element in the normal form for $c_1$ is in $A_i$ (i.e. $\beta_{1(m+1)}$ is trivial). Then 
 since $\delta_1\beta_{11}\alpha_{11}\ldots\beta_{1m}\alpha_{1m}d^{n_k}$ has the same normal form as 
 $\delta_1h_*^k(\beta_{11})\alpha_{11}\ldots h_*^k(\beta_{1m})\alpha_{1m}$ we have that $\alpha_{1m}d^{n_k}=b^{p_k}\alpha_{1m}$ 
 for some $b$ in the boundary subgroup $\langle [e_4,e_5]\rangle$. But then again for arbitrarily large $k$, $\beta_{1m}b^{p_k}$ and $h_*^k(\beta_{1m})$ 
 should represent the same right coset with respect to the boundary subgroup, a contradiction.
 
 \end{itemize}

 \item[(iv)] Let $\gamma=[(c_1,c_2,c_3)]_{E_{4_{p,q}}}$. Recall that $\gamma$ is determined by the double coset 
 $C^p_{\F_{2i+5}}(c_1)\cdot c_2\cdot C^q_{\F_{2i+5}}(c_3)$. We can easily see, by the proof of Case (iii), that 
 $c_1,c_3\in A\setminus\{1\}$. Suppose, for the sake of contradiction, that $c_2\not\in A$. If we repeat the arguments of Case (iii), 
 we get a normal form, $\delta_2\gamma_{21}\ldots\gamma_{2m}$, for $c_2$ (with respect to the above 
 mentioned amalgamated free product) and we reach the conclusion that for arbitrarily large $k$, 
 $h_*^k(c_2)= b_1^{q_k}\cdot c_2\cdot b_2^{p_k}$ for some elements $b_1,b_2\in A$. Consider the normal form, 
 $\delta_2\beta_{21}\alpha_{21}\ldots\beta_{2m}\alpha_{2m}\beta_{2(m+1)}$, for $c_2$ where $\alpha_{2j}$ 
 belongs to $A_i$. The situation is completely identical with Case (iii), we still give the arguments for completeness. We take further cases:
 
 \begin{itemize}
 \item Suppose $c_2$ is an element in $\langle e_4,e_5 \rangle \setminus \langle [e_4,e_5]\rangle$. 
 First note that $b_1, b_2$ must be in the boundary subgroup $\langle [e_4,e_5]\rangle$, otherwise 
 $b_1^{q_k}\cdot c_2\cdot b_2^{p_k}$ has different normal form from $h_*^k(c_2)$. This easily contradicts 
 Theorem \ref{Pseudo}(ii), since it implies that $\{B.h_*^k(c_2).B \ | \ k\in I\}$, for some infinite subset $I$ of $\omega$,  
 is finite. 
 
 \item Suppose the last element in the normal form for $c_2$ is not in $A_i$ (i.e. $\beta_{1(m+1)}$ is not trivial). 
 Then $b_2$ must belong to the boundary subgroup $\langle [e_4,e_5]\rangle$, otherwise $b_1^{p_k}c_2b_2^{q_k}$ has different normal form 
 from $h_*^k(c_2)$. Now as before we have that $\beta_{1(m+1)}b_2^{q_k}$ 
 represent the same right coset with respect to the boundary subgroup as $h_*^k(\beta_{1(m+1)})$, a contradiction.
 
 \item Suppose the last element in the normal form for $c_2$ is in $A_i$ (i.e. $\beta_{1(m+1)}$ is trivial). 
 Since $b_1^{p_k}\delta_2\beta_{21}\alpha_{21}\ldots\beta_{2m}\alpha_{2m}b_2^{q_k}$ has the same normal form as 
 $\delta_2h_*^k(\beta_{21})\alpha_{21}$ $\ldots h_*^k(\beta_{2m})\alpha_{2m}$ we have that $\alpha_{2m}b_2^{q_k}=b^{l_k}\alpha_{2m}$ 
 for some element $b$ in the boundary subgroup $\langle [e_4,e_5]\rangle$. Now if $m>1$, the proof is identical to the analogous part 
 in the proof of Case (iii). If $m=1$,  
 then $c_2$ has the form $\delta_2\beta\alpha$. Thus, $b_1^{p_k}\delta_2\beta b^{l_k}\alpha$ has the same normal form 
 as $\delta_2h_*^k(\beta)\alpha$, so $b_1^{p_k}$ must be in the boundary subgroup. But then again since $b_1^{p_k}\delta_2$ 
 cannot change the coset we have that $\beta b^{l_k}$, 
 $h_*^k(\beta)$ represent the same right coset with respect to the boundary subgroup, contradicting Theorem \ref{Pseudo}(ii).
 \end{itemize}
\end{itemize}

\end{proof}

\begin{lemma}\label{BasicAlgebraicClosure4}
Suppose $\gamma\in acl^{eq}(e_1,e_2,a_0,a_1,\ldots,a_i,a_{i+2})\cap \F_{2i+7}^{we}$, for some $i<n$. Then: 
\begin{itemize}
 \item if $\gamma$ is a real element, then $\gamma\in \langle e_1,e_2,$ $a_0,a_1,\ldots,a_{i+2}\rangle$;
 \item if $\gamma=[\bar{c}]_E$ for some basic equivalence relation $E$, then there is 
 $\bar{d}\in\langle e_1,e_2,a_0,a_1,$ $\ldots,a_{i+2}\rangle$ with  $\bar{c}\sim_E\bar{d}$.
\end{itemize}  
\end{lemma}
\begin{proof}
In this case we have that 
$$\langle e_1,e_2,a_0,a_1,\ldots,a_i,a_{i+2}\rangle=\langle e_1,e_2,e_3,\ldots,[e_{2i+2},e_{2i+3}],
[e_{2i+4},e_{2i+5}][e_{2i+6},e_{2i+7}]\rangle$$

The proof is identical to the proof of Lemma \ref{BasicAlgebraicClosure2} and is left to the reader. A hint is to use 
homeomorphisms of $\Sigma_{1,1}$ as well as homeomorphisms of $\Sigma_{2,1}$ (see figure \ref{Fig2}). Note that 
the Euler characteristic of $\Sigma_{2,1}$ is $-3$, thus it carries a pseudo-Anosov.
\end{proof}

\begin{figure}[ht!]
\centering
\includegraphics[width=.9\textwidth]{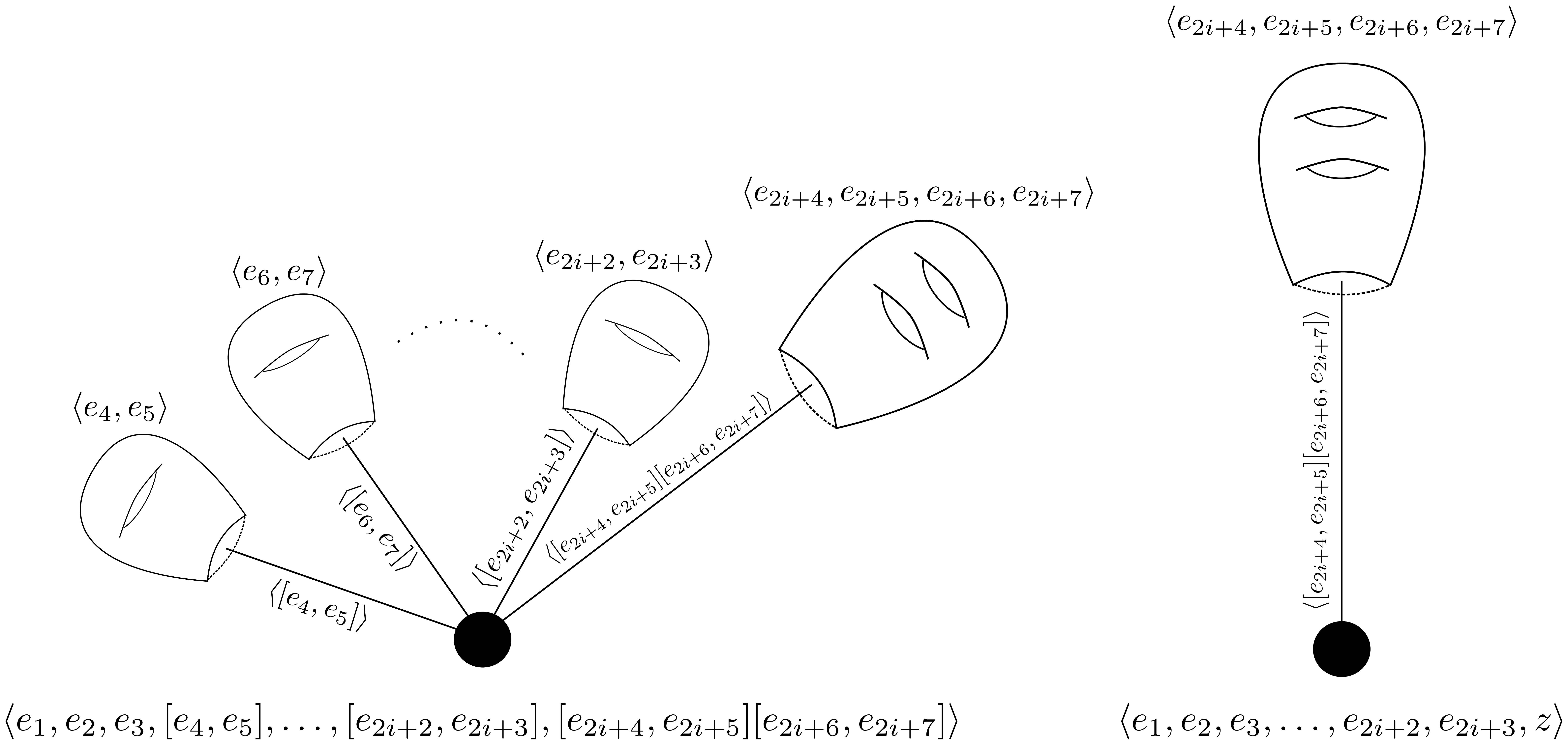}
\caption{A series of amalgamated free products (left-side). An amalgamated free product (right-side)}\label{Fig2}
\end{figure}

We are now ready to verify the fourth requirement of Definition \ref{Ample}.

\begin{prop} \label{D4}
Let $0\leq i<n$. Then $acl^{eq}(e_1,e_2,a_0,\ldots,a_i,a_{i+1})\cap acl^{eq}(e_1,e_2,a_0,$ $\ldots,a_i,a_{i+2})=acl^{eq}(e_1,e_2,a_0,\ldots,a_i)$.
\end{prop}
\begin{proof}
We denote by $A$ (respectively $B$) the group $\langle e_1,e_2,e_3,[e_4,e_5],\ldots,[e_{2i+4},e_{2i+5}]\rangle$
(respectively $\langle e_1,e_2,e_3,[e_4,e_5],\ldots,[e_{2i+4},e_{2i+5}][e_{2i+6},e_{2i+7}]\rangle$). Let $\gamma\in acl^{eq}(A)\cap acl^{eq}(B)$. 
Then there exists 
an equivalence relation $E$ and a tuple $\bar{c}$ consisting of elements of $\F_{2i+7}$ such that $\gamma=[\bar{c}]_E$.  
Thus, by Theorem \ref{Elim}, we have that $R_E(\bar{c},\F_{2i+7}^{eq})=\{\bar{\gamma}_1,\ldots,\bar{\gamma}_k\}$. Note 
that each $\bar{\gamma}_i$ is algebraic over $\F_2\gamma$, thus they all belong to $acl^{eq}(A)\cap acl^{eq}(B)$. 
Now let $\beta$ be an element of the tuple $\bar{\gamma}_i$ for some $i\leq k$. We take cases for $\beta$. 
\begin{itemize}
 \item[(i)] Suppose $\beta\in \F_{2i+7}$. Then by Lemmata \ref{BasicAlgebraicClosure2},\ref{BasicAlgebraicClosure4} we have that 
 $\beta\in A\cap B$. Consider the free splitting of $\F_{2i+7}$ as 
 $\langle e_1,e_2,\ldots,e_{2i+3}\rangle*\langle e_{2i+4},\ldots,e_{2i+7}\rangle$. 
 Let $c_1b_1\ldots c_nb_nc_{n+1}$ be the normal form for $\beta$ with respect to this splitting of $\F_{2i+7}$ where   
 $c_i\in \langle e_1,e_2,\ldots,e_{2i+3}\rangle$ and $b_i\in\langle e_{2i+4},\ldots,e_{2i+7}\rangle$ for $i\leq n$. 
 Since $\beta\in A\cap B$ we must have that 
 $b_i\in \langle [e_{2i+4},e_{2i+5}]\rangle\cap\langle [e_{2i+4},e_{2i+5}][e_{2i+6},e_{2i+7}]\rangle$, 
 but then the $b_i$'s are trivial and $\beta\in \langle e_1,e_2,e_3,\ldots,$ $[e_{2i+2},e_{2i+3}]\rangle$ as we wanted.
 
 \item[(ii)] Suppose $\beta=[c]_E$ for $E=E_1$. Then by Lemmata \ref{BasicAlgebraicClosure2},\ref{BasicAlgebraicClosure4},  
 $c$ can be conjugated to an element in $A$ and to an element in $B$. Consider the free splitting of $\F_{2i+7}$ as 
 $\langle e_1,e_2,\ldots,e_{2i+3}\rangle*\langle e_{2i+4},\ldots,e_{2i+7}\rangle$. Note that $A$ (respectively $B$) 
 inherits the following free splitting $\langle e_1,e_2,e_3,[e_4,e_5]\ldots,$ $[e_{2i+2},e_{2i+3}]\rangle*\langle [e_{2i+4},e_{2i+5}]\rangle$ 
 (respectively $\langle e_1,e_2,e_3,$ $[e_4,e_5]\ldots,$ $[e_{2i+2},e_{2i+3}]\rangle*\langle [e_{2i+4},e_{2i+5}][e_{2i+6},e_{2i+7}]\rangle$) from 
 the above splitting of $\F_{2i+7}$. 
 But, any cyclically reduced form with respect to the free splitting of $A$ (respectively $B$) is a cyclically reduced form 
 with respect to the free splitting of $\F_{2i+7}$. Thus, if $c_A:=c_1b_1\ldots c_nb_n$ (respectively $c_B:=d_1f_1\ldots d_mf_m$) 
 is the conjugate of $c$ in $A$ (respectively $B$), $c_A$ is a cyclic permutation of $c_B$, and as 
 in Case (i), we have that $c_A$ belongs to $\langle e_1,e_2,e_3,\ldots,[e_{2i+2},e_{2i+3}]\rangle$, thus $c$ can be conjugated 
 to $\langle e_1,e_2,e_3,\ldots,[e_{2i+2},$ $e_{2i+3}]\rangle$.

 \item[(iii)] Suppose $\beta=[(c_1,c_2)]_E$ for $E=E_{2_m}$. Then  by Lemmata \ref{BasicAlgebraicClosure2},\ref{BasicAlgebraicClosure4} 
 we have that there exist $c_{21}\in A$ and $c_{22}\in B$ 
 such that $C(c_2)=C(c_{21})=C(c_{22})$. Thus there are $c$ and $k_1,k_2,k_3\in\Z$ such that $c^{k_1}=c_2$, $c^{k_2}=c_{21}$, and $c^{k_3}=c_{22}$. 
 But since $A,B$ are closed under taking roots, we have that $c\in A\cap B$. Thus, as in case (i), we have that 
 $c\in \langle e_1,e_2,e_3,\ldots,[e_{2i+2},e_{2i+3}]\rangle$ and so is $c_2$. 
 
 Again by Lemmata \ref{BasicAlgebraicClosure2},\ref{BasicAlgebraicClosure4} we have that there exist $c_{11}\in A$ and $c_{12}\in B$ 
 such that $c_1\cdot C(c_2)^m=c_{11}\cdot C(c_2)^m=c_{12}\cdot C(c_2)^m$. Therefore, $c_{11}^{-1}c_1\in C(c_2)^m$ (respectively 
 $c_{12}^{-1}c_1\in C(c_2)^m$), thus $c_1\in A\cap B$ and as in Case (i) we have that    
 $c_1\in\langle e_1,e_2,e_3,\ldots,[e_{2i+2},e_{2i+3}]\rangle$, as we wanted.
 \item[(iv)] Suppose $\beta=[(c_1,c_2,c_3)]_E$ for $E=E_{4_{m,n}}$. The proof is identical to the previous case.
\end{itemize}
This shows that $\gamma$ is in $acl^{eq}(e_1,e_2,a_0,\ldots,a_i)$ and the proof is concluded.
\end{proof}

Putting everything together, Propositions \ref{D1},\ref{D2},\ref{D3}, and \ref{D4} show that the sequence $(a_i)_{i<\omega}$ witnesses 
$n$-ampleness in the theory of non abelian free groups for any $n<\omega$.

\begin{thm}
$T_{fg}$ is $n$-ample for any $n<\omega$.
\end{thm}

Our method produces a family of witnessing examples to ampleness.  
One can easily see that for each $k\geq 1$ a sequence of the form: 

$$a_0=e_3$$
$$a_{i+1}=a_i[e_{2ki+3+1},e_{2ki+3+2}][e_{2ki+3+3},e_{2ki+3+4}]\dots[e_{2ki+3+(2k-1)},e_{2ki+3+2k}]$$

witnesses $n$-ampleness for any $n<\omega$. And the same is true for the following ``non-orientable'' witnessing 
family of sequences (for $k\geq 3$): 

$$a_0=e_3$$
$$a_{i+1}=a_ie^2_{ki+3+1}e^2_{ki+3+2}\ldots e^2_{ki+3+k-1}e^2_{ki+3+k}$$

\begin{rmk} 
Let us also remark that there is no harm starting the recursive definition of our witnessing sequence with $a_0=e_1$. 
The point of not doing so is that we prefer to use Theorem \ref{Elim}, instead of Theorem \ref{WeakElim}.
\end{rmk}

\appendix
\section{Appendix}
We show that the ``basic'' equivalence relation induced by conjugation (see Definition \ref{Imaginaries}) cannot be (geometrically) eliminated 
(i.e. there exists an equivalence class which is not interalgebraic with any finite ``real'' tuple) in the theory of the free group,  
strengthening Theorem 2.1 in \cite{SelIm}. 
We also show that the ``basic'' equivalence relations $E_{2_m},E_{3_m}, E_{4_{m,n}}$ in Definition \ref{Imaginaries} cannot be eliminated 
giving alternative proofs of Theorems 2.2 and 2.3 of the same preprint mantioned above. In comparison with Sela's proofs, that uses  
the existence and properties of ``Diophantine envelopes'' for definable sets, we use the 
(also highly non-trivial) result (see \cite{Sel6,KharlampovichMyasnikov}) that the following chain of groups is elementary:

$$\F_2\prec\F_3\prec\ldots\prec\F_n\prec\ldots$$

\begin{thm}\label{Conjugation}
Fix $n\geq 2$. Then for any finite tuple $\bar{a}\in \F_{n+1}$, we have that $[e_{n+1}]_{E_1}$ is not interalgebraic with $\bar{a}$ over $\F_n$. 
\end{thm}
\begin{proof}
Suppose for the sake of contradiction that there is $\bar{a}\in\F_{n+1}$ such that $[e_{n+1}]_{E_1}$ is 
interalgebraic with $\bar{a}$ over $\F_n$. It is not hard to see that $\bar{a}\in\F_{n+1}\setminus\F_n$ 
(otherwise we can fix $\bar{a}$ and send $e_{n+1}$ to $e_{n+i}$ for $1<i<\omega$).\\
{\em Claim:} Let $b\in\F_n$ and $f_b\in Aut_{\F_n}(\F_{n+1})$ be defined by $f_b(e_{n+1})=e_{n+1}^b$. Then 
$\bar{a}$ has infinite orbit under $\langle f^l_b | l<\omega \rangle$. \\ 
{\em Proof of Claim:} Let $c=e_{n+1}^{i_1}w_1(e_1,\ldots,e_n)e_{n+1}^{i_2}\ldots 
e_{n+1}^{i_k}w_k(e_1,\ldots,e_n)e_{n+1}^{i_{k+1}}$ be the normal form of an element in the tuple $\bar{a}$ which is moreover  
in $\F_{n+1}\setminus\F_n$ with respect to $\F_n*\langle e_{n+1}\rangle$. Then $f_b^l(c)=b^le_{n+1}^{i_1}b^{-l}w_1(e_1,\ldots,$ $e_n)b^le_{n+1}^{i_2}b^{-l}\ldots 
b^le_{n+1}^{i_k}b^{-l}w_k(e_1,\ldots,e_n)$ $b^le_{n+1}^{i_{k+1}}b^{-l}$
and the claim follows easily.
 
Now, since $f^l_b$ fixes the conjugacy class of $e_{n+1}$, we have $\bar{a}\not\in acl^{eq}(\F_n, [e_{n+1}]_{E_1})$, a contradiction.
  
\end{proof}

We continue by proving that no basic imaginary can be eliminated.

\begin{thm}
Fix $n\geq 2$. Let $E$ be a basic equivalence relation (see Definition \ref{Imaginaries}). 
Then there exists an equivalence class $[\bar{b}]_E$ such that for any $\bar{a}\in\F_{\omega}$, 
we have that $[\bar{b}]_E$ and $\bar{a}$ are not interdefinable over $\F_n$.
\end{thm}
\begin{proof}
We take cases according to the basic equivalence relation $E$.
\begin{itemize}
\item[(i)] Let $E=E_1$. Then the result follows from Theorem \ref{Conjugation};
\item[(ii)] Let $E=E_{2_m}$. Then we consider the class $[(e_{n+1},e_{n+2})]_E$, which is determined by the left coset $e_{n+1}\cdot C_{\F_{n+2}}^m(e_{n+2})$. 
Suppose, for the sake of contradiction, that there is $\bar{a}\in\F_{\omega}$ such that $[(e_{n+1},e_{n+2})]_E$ is interdefinable with  
$\bar{a}$ over $\F_n$. Then as in the proof of Theorem \ref{Conjugation} 
we must have that $\bar{a}\in\F_{n+2}\setminus\F_{n+1}$. But then the automorphism of $\F_{n+2}$ 
fixing $\F_{n+1}$ and sending $e_{n+2}\mapsto e_{n+2}^{-1}$ fixes $[(e_{n+1},e_{n+2})]_E$ and moves 
$\bar{a}$, a contradiction;
\item[(iii)] Let $E=E_{4_{m,n}}$. Then we consider the class $[(e_{n+1},e_{n+2},e_{n+1})]_E$ and the result follows as above. 
\end{itemize}
\end{proof}


\begin{thebibliography}{10}

\bibitem{BesDeg}
Bestvina M.
\newblock Degenarations of the hyperbolic space.
\newblock Duke Math J 1988; 56: 143--161.

\bibitem{EvAmp}
Evans D.
\newblock Ample dividing.
\newblock J Symbolic Logic 2003; 68: 1385--1402.

\bibitem{Primer}
Farb B, Margalit D.
\newblock A Primer on Mapping Class Groups.
\newblock New Jersey, USA: Princeton University Press, 2011.

\bibitem{ThurstonSurface}
Fathi A, Laudenbach F, Po\'{e}naru V  (translated by Kim DM and Margalit D).
\newblock Thurston's work on surface.
\newblock Princeton, NJ, USA: Princeton University Press, 2012.


\bibitem{HatcherAlgTop}
Hatcher A.
\newblock Algebraic Topology.
\newblock Cambridge, UK: Cambridge University Press, 2002.


\bibitem{HrushovskiSMSet}
Hrushovski E.
\newblock A new strongly minimal set.
\newblock Ann Pure Appl Logic 1993; 62: 147--166.

\bibitem{WeakNormal}
Hrushovski E, Pillay A.
\newblock Weakly normal groups.
\newblock In: The Paris logic group, editors. Studies in 
  Logic and Foundations of Mathematics - Logic colloquium '85; July 1985; Orsay, France. Amsterdam: North-Holland, 1987, pp.233--244.

\bibitem{KharlampovichMyasnikov}
Kharlampovich O, Myasnikov A.
\newblock {Elementary theory of free nonabelian groups}.
\newblock J Algebra 2006; 302: 451--552.

\bibitem{LouderPerinSklinosTowers}
Louder L, Perin C, Sklinos R. Hyperbolic towers and independent generic sets in the theory of free groups, 
Notre Dame J. Form. Log. 2013; 54 (3-4): 521--539.

\bibitem{LyndonSchupp}
Lyndon RC, Schupp PE.
\newblock Combinatorial Group Theory.
\newblock New York, NJ, USA: Springer-Verlag, 1977. 

\bibitem{MagnusKarrassSolitar}
Magnus W, Karrass A, Solitar D.
\newblock Combinatorial Group Theory : Presentations of Groups in Terms of Generators and Relations.
\newblock New York, NJ, USA: Dover, 1976.

\bibitem{MarkerModelTheory}
Marker D.
\newblock Model Theory: An Introduction.
\newblock New York, NJ, USA: Springer, 2002. 

\bibitem{NielsenI}
Nielsen J.
\newblock {Untersuchungen zur {T}opologie der geschlossenen zweiseitigen {F}lachen {I}}.
\newblock Acta Math 1927; 50: 189--358.

\bibitem{NielsenII}
Nielsen J.
\newblock {Untersuchungen zur {T}opologie der geschlossenen zweiseitigen {F}lachen {II}}.
\newblock Acta Math 1929; 53: 1--76.

\bibitem{NielsenIII}
Nielsen J.
\newblock {Untersuchungen zur {T}opologie der geschlossenen zweiseitigen {F}lachen {III}}.
\newblock Acta Math 1932; 58: 87--176.

\bibitem{OuldTentAmple}
Ould Houcine A, Tent K. 
\newblock {Ampleness in the free group}. 
\newblock preprint, available at \href{http://arxiv.org/abs/1205.0929}{arXiv:1205.0929v2 [math.GR]}.


\bibitem{PaulinGromov}
Paulin F.
\newblock {Topologie de Gromov \'equivariante, structures hyperboliques et
  arbres r\'eels}.
\newblock Invent Math 1988; 94: 53--80.

\bibitem{PennerPseudo}
Penner RC.
\newblock {A construction of pseudo-Anosov homeomorphisms}.
\newblock Trans Am Math Soc 1988; 310: 179--197.

\bibitem{PerinSklinosForking}
Perin C, Sklinos R.
\newblock {Forking and JSJ decompositions in the Free Group}.
\newblock To appear in the J Eur Math Soc.

\bibitem{PillFMR}
Pillay A.
\newblock The geometry of forking and groups of finite morley rank.
\newblock J Symbolic Logic 1995; 60: 1251--1259.

\bibitem{PillayStability}
Pillay A.
\newblock Geometric Stability Theory.
\newblock Oxford, UK: Oxford University Press, 1996.

\bibitem{PiAmp}
Pillay A.
\newblock {A Note on CM-Triviality and The Geometry of Forking}.
\newblock J Symbolic Logic 2000; 65: 474--480.

\bibitem{PillayForking}
Pillay A.
\newblock Forking in the free group.
\newblock J Inst Math Jussieu 2008; 7: 375--389.

\bibitem{PillayGenericity}
Pillay A.
\newblock On genericity and weight in the free group.
\newblock P Am Math Soc 2009; 137: 3911--3917.

\bibitem{PoizatModelTheory}
Poizat B (translated by Klein MG).
\newblock A Course in Model Theory: An Introduction to Contemporary Mathematical Logic.
\newblock New York, NJ, USA: Springer, 2000.

\bibitem{PoizatGenericAndRegular}
Poizat B.
\newblock Groupes stables, avec types g\'en\'eriques r\'eguliers.
\newblock J Symbolic Logic 1983; 48: 339--355.

\bibitem{PoizatStableGroups}
Poizat B (translated by Klein MG). 
\newblock Stable Groups.
\newblock {Providence, RI, USA: Amer Math Soc, 2001}.


\bibitem{Sel5}
Sela Z.
\newblock {Diophantine geometry over groups {$\rm V\sb 1$}. Quantifier
  elimination {I}}.
\newblock Israel J Math 2005; 150: 1--197.

\bibitem{Sel5bis}
Sela Z.
\newblock {Diophantine geometry over groups {$\rm V\sb 2$}. Quantifier
  elimination {II}}.
\newblock Geom Funct Anal 2006; 16: 537--706.

\bibitem{Sel6}
Sela Z.
\newblock {Diophantine geometry over groups VI: The elementary theory of free
  groups}.
\newblock Geom Funct Anal 2006; 16: 707--730.

\bibitem{SelaStability}
Sela Z.
\newblock {Diophantine geometry over groups VIII: Stability}.
\newblock Ann Math 2013; 177: 787--868.

\bibitem{SelIm}
Sela Z.
\newblock {Diophantine geometry over groups {IX}: Envelopes and Imaginaries}. 
\newblock preprint, available at \url{http://www.ma.huji.ac.il/~zlil/}.

\bibitem{SerreTrees}
Serre J-P.
\newblock {Trees}.
\newblock {New York, NJ, USA: Springer, 2003}.

\bibitem{SklinosAmple}
Sklinos R.
\newblock {A note on ampleness in the theory of non Abelian free groups}.
\newblock preprint, available at \href{http://arxiv.org/abs/1205.4662}{arXiv:1205.4662v2 [math.LO]}.


\bibitem{ThurstonPseudo}
Thurston WP.
\newblock On the geometry and dynamics of diffeomorphisms of surface.
\newblock B Am Math Soc 1988; 19: 417--431.

\end{thebibliography}
\end{document}